\documentclass[a4paper,12pt]{amsart}
\pdfoutput=1

\usepackage[english]{babel}
\usepackage{bbold}
\usepackage{stmaryrd}
\usepackage{verbatim}
\usepackage[utf8]{inputenc}
\usepackage{amsmath}
\numberwithin{equation}{subsection}
\usepackage{amssymb}
\usepackage{amsthm}
\usepackage{graphicx}
\usepackage{mathtools}
\usepackage[T1]{fontenc}
\usepackage{amsthm}
\usepackage{setspace}
\usepackage{amsfonts}
\usepackage{eucal}
\usepackage{mathrsfs}

\usepackage{pictexwd,dcpic}
\usepackage{tikz}
\usepackage{tqft}
\usepackage{pigpen}
\usepackage{xcolor}
\usepackage{pdfpages}
\usepackage{geometry}

\usepackage{hyperref}
\hypersetup{
	colorlinks=true,
	linkcolor=blue,
	filecolor=magenta,      
	urlcolor=cyan,
}

\usepackage{thm-restate}
\usepackage{tikz-cd}
\usepackage{float}

\usepackage{blkarray}

\usepackage[all,cmtip]{xy}
\UseComputerModernTips

\newtheorem*{theorem*}{Theorem}
\newtheorem{theorem}{Theorem}[section] 
\newtheorem{lemma}[theorem]{Lemma}
\newtheorem*{lemma*}{Lemma}      

\newtheorem{proposition}[theorem]{Proposition}

\newtheorem*{conj*}{Conjecture}

\theoremstyle{definition}
\newtheorem{definition}[theorem]{Definition}
\newtheorem*{definition*}{Definition}

\newtheorem{them}{Theorem}

\newtheorem{themprime}{Theorem}

\newtheorem{thei}{Theorem}

\theoremstyle{remark}
\newtheorem{remark}[theorem]{Remark}
\numberwithin{equation}{section}

\theoremstyle{definition}

\newtheorem{conjecture}{\bf Conjecture}

\newcommand{\C}{{\mathbb{C}}}
\newcommand{\N}{{\mathbb{N}}}
\newcommand{\PP}{{\mathbb{CP}}}

\newcommand{\Z}{{\mathbb{Z}}}

\newcommand{\M}{{\mathcal{M}}}

\newcommand{\cp}{\mathcal{P}}
\newcommand{\x}{\boldsymbol{x}}
\newcommand{\z}{\boldsymbol{z}}
\newcommand{\Mkm}{\mathcal{M}_{k,m}}
\newcommand{\Fkm}{F_{k,m}}
\newcommand{\y}{\boldsymbol{y}}

\newcommand{\ud}{\underline}
\newcommand{\ob}{\llbracket}
\newcommand{\cb}{\rrbracket}

\DeclareMathOperator{\Span}{Span}

\DeclareMathOperator{\Spec}{Spec}
\DeclareMathOperator{\Ker}{Ker}

\DeclareMathOperator{\Id}{Id}

\DeclareMathOperator{\Imm}{Im}

\renewcommand{\b}{\beta}

\newcommand{\om}{\omega}

\newcommand{\D}{\mathbb{D}}
\newcommand{\p}{\pi}

\title[]
{Fixed points of Koch's maps } 

\author[ Van Tu Le ]{Van Tu Le}
\address{Van Tu Le\\ Dipartimento di Matematica\\
	Università degli Studi di Roma "Tor Vergata"\\Via della Ricerca Scientifica 1 - 00133 Roma\\ Italy} 
\email{levantu.hp@gmail.com}
\date{\today}

\begin{document}
	\begin{abstract}
		We study endomorphisms constructed by Sarah Koch in \cite{koch2013teichmuller} and we focus on the eigenvalues of the differential of such maps at its fixed points. In \cite{koch2013teichmuller}, to each post-critically finite unicritical polynomial, Koch associated a post-critically algebraic endomorphism of $\PP^k$. Koch showed that the eigenvalues of the differentials of such maps along periodic cycles outside the post-critical sets have modulus strictly greater than $1$. In this article, we show that the eigenvalues of the differentials at fixed points are either $0$ or have modulus strictly greater than $1$. This confirms a conjecture proposed by the author in his thesis. We also provide a concrete description of such values in terms of the multiplier of a unicritical polynomial. 
	\end{abstract}
\maketitle
\section{Introduction}
Let $M$ be either $\C^n$ or $\PP^n$ and $f \colon M \to M$ be a holomorphic endomorphism. Denote by $f^{\circ m} = f \circ f \circ \ldots \circ f$ the $m$-th composition of $f$.  A point $z \in M$ is called a \textit{preperiodic point} of \textit{preperiod} $k$ and of \textit{period} $m$ if $f^{\circ(k+m)}(z) =f^{\circ k}(z) $ and $k,m$ are the smallest integers satisfying such a property. A preperiodic point of preperiod $0$ is called a \textit{periodic point}. A periodic point of period $1$ is called a \textit{fixed point}. Given a periodic point $z$ of period $m$, a value $\lambda \in \C$ is called \textit{an eigenvalue of $f$ along the orbit of $z$} (or \textit{at the fixed point $z$}) if $\lambda$ is an eigenvalue of the differential $D_z f^{\circ m} \colon T_{z} M \to T_{z} M $. 

A point $z \in M$ is called \textit{a critical point} if the differential $D_{z} f \colon T_{z} M \to T_{f(z)} M$ is not invertible. The set $C(f)$ containing all critical points of $f$ is called the \textit{critical set} of $f$. The set
\[
PC(f) \coloneqq \bigcup\limits_{j \ge 1} f^{\circ j} (C(f))
\]
is called the \textit{post-critical set} of $f$. The endomorphism $f$ is called \textit{post-critically algebraic} if $PC(f)$ is an algebraic set of codimension one in $M$. When $\dim M = 1$, post-critically algebraic rational maps are called \textit{post-critically finite} rational maps. 

The family of post-critically finite rational maps is one of the most important families of maps in the theory of one dimensional complex dynamics. In higher dimension, post-critically algebraic endomorphisms are interesting family of maps since many results, which are well-known for post-critically finite rational maps, remain unknown. We refer to \cite{rong2008fatou},\cite{astorg2018dynamics},\cite{ingram2019post},\cite{gauthier2019geometric}, \cite{ji2020structure}, \cite{van2020periodic} for some recent studies about post-critically algebraic endomorphisms. In this article, we focus on the following conjecture proposed by the author in his thesis \cite{le2020dynamique}.

\begin{conjecture}\label{conj 1}
		Let $f$ be a post-critically algebraic endomorphism of $\PP^n, n \ge 2$ of degree $d\geq 2$ and $\lambda$ be an eigenvalue of $f$ along a periodic cycle. Then either $\lambda=0$ or $|\lambda|>1$.
\end{conjecture}

The conjecture has been verified by the author in the case $n = 2$ and in the case in any dimension with the periodic cycles outside the post-critical set. In this article, we shall verify the conjecture for the family of post-critically algebraic endomorphisms associated to unicritical polynomials constructed by Sarah Koch in \cite{koch2013teichmuller}, or \textit{Koch maps} for short.

We shall now describe the family of Koch maps we want to study and we refer to \cite{koch2013teichmuller} for the original construction. Throughout this article, we fix
\[
d \in \mathbb{N} , d \ge 2 \text{ and } \beta^d = 1, \beta \neq 1.
\]
The maps $\{G_{k,m} \colon \C^{k+m-1} \to \C^{k+m-1}\mid (k,m) \in \N \times \N^*\}$ constructed by Sarah Koch are of the following forms (see \cite[Proposition 6.1 - 6.2]{koch2013teichmuller}),
\begin{itemize}
	\item if $k=0$,
	\[
	G_{0,m}\colon \left(
	\begin{array}{c}
	x_1\\
	x_2\\
	\vdots\\
	x_{m-1}
	\end{array}
	\right) \mapsto \left(
	\begin{array}{c}
	- x_{m-1}^d\\
	x_1^d-x_{m-1}^d\\
	\vdots\\
	x_{m-2}^d - x_{m-1}^d
	\end{array}
	\right),
	\]
	\item if $k \neq 0$, 
	\[
	G_{k,m}\colon \left(
	\begin{array}{c}
	x_1\\
	x_2\\
	\vdots\\
	x_{k+m-1}
	\end{array}
	\right) \mapsto \left(
	\begin{array}{c}
	\left( - \frac{\b x_{k+m-1} - x_{k-1}}{\b - 1} \right)^d\\
	\left( x_1 - \frac{\b x_{k+m-1} - x_{k-1}}{\b -1 } \right)^d \\
	\vdots\\
	\left( x_{k+m-2} - \frac{\b x_{k+m-1}  - x_{k-1}}{\b-1} \right)^d
	\end{array}
	\right).
	\]
\end{itemize}
The map $G_{k,m}$ induces a holomorphic endomorphism of $\PP^{k+m-2}$ which is closely related to maps on moduli spaces used in Thurston's topological characterization of rational maps. We refer to \cite{koch2013teichmuller},\cite{koch2008new},\cite{douady1993proof} for further discussion. In \cite{koch2013teichmuller}, Koch showed that $G_{k,m}$ is post-critically algebraic. It is natural to ask whether Conjecture \ref{conj 1} is true for $G_{k,m}$. The eigenvalues of Koch maps along a periodic cycle outside the post-critical set are well understood. It is a consequence of its construction that those values has modulus strictly bigger than $1$.
\begin{thei}[Corollary 7.2 \cite{koch2013teichmuller}\label{thm:origin}
	]
Let $\mu$ be an eigenvalue of $G_{k,m}$ along a periodic cycle outside the post-critical set. Then $|\mu|>1$.
\end{thei}
We refer also to \cite{buff2017eigenvalues} for a further discussion about the arithmetics of such values. In \cite{koch2013teichmuller}, Koch asked whether we have the same conclusion for eigenvalues along a cycle inside the post-critical set. In this article, we answer this question in the positive (and hence verify Conjecture \ref{conj 1}) for the case when the cycle is a fixed point.
\begin{them}\label{thm:theoremA}
	Let $\mu$ be an eigenvalue of a map $G_{k,m}$ at a fixed point. Then, either $\mu = 0 $ or $|\mu|>1$.
\end{them}
In fact, we can have even better understanding about the values of such eigenvalues. Thanks to Theorem \ref{thm:origin}, we only need to study the eigenvalues at a fixed point inside the post-critical set of $G_{k,m}$. However, the original construction does not provide much information about fixed points inside post-critical set. In order to explain our result, let us take a closer look at $G_{k,m}$.

Let $(k,m) \in \N \times \N^*$. To a point $z \in \C^{k+m-1}$, we can associate a polynomial $P_z$ of the following form
\[
P_z(t) = \begin{cases*}
t^d - z_{m-1}^d \text{ if $k = 0$}\\
\left( t - \frac{\b z_{k+m-1} - z_{k-1}}{\b -1 } \right)^d \text{ if $k \neq 0$}.
\end{cases*}
\] 
If $z$ is a fixed point of $G_{k,m}$, then for all $1 \le i \le k+m-1$, $z_i = P_z^{\circ i}(0)$. Moreover, $ P_z$ is post-critically finite. Indeed, 
\begin{itemize}
\item if $k = 0$ then $P_z(z_{m-1})  = 0$,
\item if $k \neq 0$ then 
\[\begin{array}{ccl}
P_{z}(z_{k+m-1}) & =& \left( z_{k+m-1} - \frac{\b z_{k+m-1}  - z_{k-1}}{\b-1} \right)^d = \left( \frac{- z_{k+m-1}  + z_{k-1}}{\b-1} \right)^d\\
& = & \left( \frac{\b z_{k-1}  -\b z_{k+m-1}}{\b-1} \right)^d = P_{z} (z_{k-1})
\end{array}\]
\end{itemize} 
We shall call $P_z$ the \textit{polynomial associated to $z$} since $P_z$ plays an important role in the study of the eigenvalues of $G_{k,m}$ at a fixed point $z$. More precisely, our main result, which completes the description of eigenvalues of Koch maps at fixed points, is the following. 
\begin{themprime}\label{thm:theoremA'}
	Let $\mu$ be an eigenvalue of a map $G_{k,m}$ at a fixed point $z  = (z_1,\ldots,z_{k+m-1})$. Let $P_z$ be the polynomial associated to $z$ and $z_1$ is preperiodic of preperiod $k'$ and of period $m'$ to a cycle of multiplier $\lambda$ under $P_z$. 

Only one of the following cases happens: 
	\begin{enumerate}
		\item\label{itm:1} $\mu = 0$.
		\item\label{itm:2} $\mu$ is an eigenvalue of a map $G_{k',m'} $ at a fixed point outside the post-critical set $PC(G_{k',m'})$.
		\item\label{itm:3} We have
		\[
		\mu^m = \lambda^{\frac{m}{m'}}, \mu^{m'} \neq \lambda.
		\]
	\end{enumerate}
\end{themprime}
We can see that Theorem \ref{thm:theoremA} is a direct consequence of Theorem \ref{thm:theoremA'}. Indeed, if Case \ref{itm:1} or Case \ref{itm:2} happens, Theorem \ref{thm:theoremA} follows from Theorem \ref{thm:origin}. If Case \ref{itm:3} happens, since the polynomial $P_z$ is a {post-critically finite} polynomial, Theorem \ref{thm:theoremA} follows from the equation $\mu^m = \lambda^{\frac{m}{m'}}$ and the fact that a non-vanishing multiplier of a post-critically finite polynomial has modulus strictly bigger than $1$ (see \cite[Corollary 14.5]{milnor2011dynamics})

Let us explain briefly our approach. Instead of using the original construction of $G_{k,m}$, we introduce 

 \begin{itemize}
 	\item[-]a partial order $\preceq$ on $\N \times \N^*$,
 	\item[-] a dynamically equivalent family of maps, that we denote by \[\{F_{k,m}\colon \M_{k,m} \to \M_{k,m} , (k,m) \in \N \times \N^*\},\] where $\M_{k,m}$ is a subspace of the vector spaces of complex sequences $\C^{\N^*}$.
 \end{itemize} 
More precisely, with the convention $x_0 \coloneqq 0$, the space $\M_{k,m}$ is defined as
\[
\Mkm = \{ \x  = (x_i)_{i \ge 1} \mid \forall \, i \ge k+1, x_i = x_{i+m} \text{ and } \beta x_{k+m} - x_k = 0 \}
\]
and the map $F_{k,m} \colon \Mkm \to \Mkm $ is defined as 
\[
\Fkm(\x) = \y \Leftrightarrow \left\{
\begin{array}{l}
\b y_{k+m} - y_k = 0\\
y_i = x_{i-1}^d + y_1  \quad	\text{ for all $ i \ge 2$}.
\end{array}
\right.
\]
The construction shall be presented in details in Section \ref{sec: construction of Fkm}. For each $(k,m) \in \N \times \N^*$, the maps $\Fkm$ and $G_{k,m}$ are conjugate. Thus, to prove Theorem \ref{thm:theoremA'}, we need to study an eigenvalue $\mu$ of $\Fkm$ at a fixed point $\z \in \Mkm$. The full statement of what we can prove is Theorem \ref{thm: Koch prime}. Briefly, to each fixed point $\z$ in $\Mkm$, we associate a pair of integers $(k',m')$ and a post-critically finite polynomial $P_z$ whose critical value is preperiodic of preperiod $k'$ to a cycle of period $m'$. The associated polynomial $P_z$ for a $\z = (z_i)_{i \ge 1} \in \Mkm$ is simply $P_z(t) = t^d + z_1$. The partial order characterizes the following property of the family $\{\Fkm \}$:
\[
(k',m') \preceq (k,m) \Leftrightarrow \begin{cases*}
\M_{k',m'} \subseteq \Mkm,\\ \Fkm|_{\M_{k',m'}}  = F_{k',m'}.
\end{cases*}
\]
Moreover, we shall show that $\z \notin PC(F_{k',m'})$. Thus, if $(k',m') = (k,m)$, our fixed point $\z$ is outside the post-critical set $PC(F_{k,m})$ and we are in Case \ref{itm:2}.

 If $(k',m') \neq (k,m)$, then $\M_{k',m'}$ is a proper subset of $\Mkm$ which is invariant under $\Fkm$ and the restriction of $\Fkm$ to $\M_{k',m'}$ is exactly $F_{k',m'}$. If $\mu$ has associated eigenvectors tangent to $\M_{k',m'}$, since $\z \notin PC(F_{k',m'})$, we are again in the Case \ref{itm:2}. Otherwise, $\mu$ is the eigenvalue of the transpose $D_{\z} \Fkm^*$ of the derivative $D_{\z} \Fkm$ acting on the annihilator $\M_{k',m'}^0 = \{\omega \in \Mkm^* \mid \omega|_{\M_{k',m'}} = 0\}$. In such a case, either $\mu = 0$ or we will show that $\M_{k',m'}^0$ has a set of generators on which $D_{\z} \Fkm^*$ acts cyclically and we obtain Case \ref{itm:3} by solving a linear algebra problem.
\medskip

	\textbf{Acknowledgement:} This article is the improvement of Chapter $2$ of the author's thesis \cite{le2020dynamique}. The author is grateful to his supervisors Xavier Buff and Jasmin Raissy for their support,
	suggestions and encouragement. The author would like to thank also Valentin Huguin for his comments and useful discussions. This work is supported by the fellowship of Centre International de Math\'ematiques et d'Informatique de Toulouse (CIMI).
	\\
\section{An alternative construction of Koch maps}\label{sec: construction of Fkm}
\subsection{Construction of $F_{k,m}$ and $\M_{k,m}$}
Recall that in this article, we fix 
\[
d \in \mathbb{N} , d \ge 2 \text{ and } \beta^d = 1, \beta \neq 1
\]
Let $(k,m) \in \N \times \N^*$ and denote by $
	\mathcal{E} = \C^{\N^*}$ the vector space of complex sequences $\x = (x_i)_{i \ge 1}$. Set $\boldsymbol{0} \coloneqq (0,0,\ldots)$. Let $\mathcal{L} \subset \mathcal{E}$ be the one-dimensional subspace consisting of constant sequences,
	\[
	\mathcal{L} = \{ \x \in \mathcal{E}  \mid \forall i,j \ge 1, \quad  x_i = x_j \},
	\] and $\mathcal{H}_{k,m} \subset \mathcal{E}$ be the hyperspace of $\mathcal{E}$ defined by \[
	\mathcal{H}_{k,m} \coloneqq \{ \x \in \mathcal{E} \mid \b x_{k+m} - x_k= 0 \}.
	\]
	with the convention $x_0\coloneqq0$. In particular, when $k = 0$, \[\mathcal{H}_{0,m} = \{ \x \in \mathcal{E} \mid x_{m} = 0  \}.\]
	\begin{lemma}\label{eq:decomposition of E}
Given $(k,m) \in \N \times \N^*$, we have $\mathcal{E} = \mathcal{H}_{k,m} \oplus \mathcal{L}$.
	\end{lemma}	
	\begin{proof}
On the one hand, given $\x \in \mathcal{E}$, define $\y \in \mathcal{E}$ by \[y_i \coloneqq x_i - \kappa \text{ with } \kappa = \left\{\begin{array}{ll}
x_m & \text{ if } k = 0 \\
\frac{\b x_{k+m} - x_{k}}{\b - 1} & \text{ if } k \ge 1
\end{array}\right..\]
Then $\b y_{k+m} -y_k =0$ hence $\y \in\mathcal{H}_{k,m}$. Note that $\x - \y \in \mathcal{L}$ hence \[\mathcal{E} = \mathcal{H}_{k,m} + \mathcal{L}. \]
		
On the other hand, assume $\x \in \mathcal{H}_{k,m} \cap \mathcal{L}$. Then $\b x_{k+m} - x_{k} = 0$ and $x_{k+m} = x_{k}$. Since $\b \neq 1$, we have $x_{k} = x_{k+m} = 0$. Moreover, $\x$ is a constant sequence. Thus, $\x$ vanishes identically; that is \[{\mathcal{H}_{k,m} \cap \mathcal{L} = \{\textbf{0}\}} . \qedhere\] 
	\end{proof}Denote by \[
	\p_{k,m} \colon\mathcal{E} \to \mathcal{E}
	\]
 the projection to $\mathcal{H}_{k,m}$ parallel to $\mathcal{L}$. In particular, $\p_{k,m}(\mathcal{E}) = \mathcal{H}_{k,m}$.
Consider the map $\mathcal{Q}: \mathcal{E} \to \mathcal{E}$ defined by 
	\[
	\mathcal{Q}(\x) = \y \quad \text{ with} \quad y_1 \coloneqq0 \text{ and }y_i = x_{i-1}^d, i \ge 1.
	\]
	\begin{definition}
Given $(k,m) \in \N \times \N^*,$ the map $\Fkm \coloneqq \mathcal{E} \to \mathcal{E}$ is defined as\[
		\Fkm\coloneqq \p_{k,m} \circ \mathcal{ Q} .
		\]
	\end{definition}
We shall now study some important properties of $\Fkm$.
\subsubsection{Properties of $\Fkm$}
 \begin{lemma}\label{lem:explicit form of Fkm}
Given $(k,m) \in \N \times \N^*,$ for every $\x,\y \in \mathcal{E}$, we have
 	\[
 	\Fkm(\x) = \y \Leftrightarrow \left\{
 	\begin{array}{l}
 	\b y_{k+m} - y_k = 0\\
 y_i = x_{i-1}^d + y_1  \quad	\text{ for all $ i \ge 2$}.
 	\end{array}
 	\right.
 	\]
In particular, with the convention $x_0 \coloneqq 0 $, we have
 	\begin{equation*}
 	y_1= \left\{\begin{array}{lc}
 	-x_{m-1}^d & \text{ if } k = 0,\\
 	-\frac{\b x_{k+m-1}^d -x_{k-1}^d}{\b - 1 } & \text{ if } k \ge 1.
 	\end{array} \right.
 	\end{equation*}
 	
 \end{lemma}
 \begin{proof}

 	Assume $ \x \in \mathcal{E}$ and $\y = \pi_{k,m}(\mathcal{Q}(\x)) = \Fkm(\x)$. On the one hand, $\y \in \pi_{k,m}(\mathcal{E}) = \mathcal{H}_{k,m}$, i.e. \[\b y_{k+m} - y_k = 0.\] On the other hand, set $\z = \mathcal{Q}(\x)$, i.e. $z_1 = 0$ and $z_i = x_{i-1}^d$ for all $i \ge 1$. Then since $\y =\pi_{k,m} (\z)$, for all $i \ge 1$,
 	\[
 	y_i = z_i - \kappa	\text{ with } \kappa \coloneqq \left\{\begin{array}{ll}
 	 	z_m & \text{ if } k = 0 \\
 	 	\frac{\b z_{k+m} - z_{k}}{\b -1} & \text{ if } k \ge 1
 	 	\end{array}\right..\]
 	In particular, \[y_1 = z_1 - \kappa= -\kappa =\left\{\begin{array}{lc}
 	-x_{m-1}^d & \text{ if } k = 0,\\
 	-\frac{\b x_{k+m-1}^d - x_{k-1}^d}{\b - 1} & \text{ if } k \ge 1.
 	\end{array} \right.\]whence for all $i \ge 2$, $y_i = z_i + y_1$, i.e. 
\[ 	y_i = x_{i-1}^d + y_1.\]
 	
 	Conversely, assume $\x,\y \in \mathcal{E}$ such that $\b y_{k+m} - y_k =0$ and for all $i \ge 2$, $y_i =x_{i-1}^d + y_1$. In particular, $\y \in \mathcal{H}_{k,m}$. Set $\z = \y - \mathcal{Q}(\x)$. Then for all $i \ge 2$, 
 	\[
 	z_i = y_i - x_{i-1}^d = y_1
 	\]
Moreover, $z_1= x_0^d + y_1= y_1$. Hence $\z \in \mathcal{L}$. In other words, \[\y = \pi_{k,m} (\mathcal{Q}(\x)) = \Fkm(\x). \qedhere\] 
 \end{proof}

	Although $\Fkm$ is defined on a vector space of infinite dimension, we will now see that the dynamics of $\Fkm$ is captured entirely by some finite dimensional vector space. Given a sequence $\x \in \mathcal{E}$, $\x$ is preperiodic of preperiod $k$ to a cycle of period $m$ if for all $i \ge k+1,  x_{i} = x_{i+m}$ and $k,m$ are the smallest integers satisfying such a property. Given ${(k,m) \in \N\times \N^*}$, let $\mathcal{P}_{k,m} \subset \mathcal{E}$ be the subspace of preperiodic sequences of preperiod at most $k$ to a cycle of period dividing $m$, i.e.
	\[
	\mathcal{P}_{k,m}\coloneqq \{ \x \in \mathcal{E} \mid x_{i+m} = x_{i} \text{ for } i \ge k+1  \}.
	\]
 Since sequences in $\mathcal{P}_{k,m}$ are uniquely determined by the first $k+m$ entries, the vector space $\mathcal{P}_{k,m}$ has finite dimension $k+m$.
\begin{definition}
Given $(k,m) \in \N \times \N^*,$ define \[\Mkm \coloneqq \mathcal{P}_{k,m} \cap \mathcal{H}_{k,m}.\]
\end{definition}
Note that the constant sequence $(1,1,\ldots)$ is in $\mathcal{P}_{k,m} \setminus \mathcal{H}_{k,m}$ and that $\mathcal{H}_{k,m}$ has codimension one in $\mathcal{E}$. Hence, $\M_{k,m}$ is a vector space of dimension $k+m-1$. The following two lemmas show the importance of $ \M_{k,m}$ to the dynamics of $\Fkm$.
	
\begin{lemma}\label{lem: moduli map is nondegenerate}
		We have $F_{k,m}(\M_{k,m}) = \M_{k,m}$ and $F_{k,m}\colon\M_{k,m} \to \M_{k,m}$ is a nondegenerate homogeneous map of degree $d$.
	\end{lemma}
	\begin{proof}
Let us first prove that $\mathcal{Q}(\Mkm) \subseteq \mathcal{P}_{k,m}.$ Assume $\x \in \M_{k,m}$.  Since $\x \in \mathcal{P}_{k,m} \cap \mathcal{H}_{k,m}$, $\x$ has preperiod $k$ and period dividing $m$ and $\b x_{k+m} - x_{m} =0 $. In particular, since $\b^d = 1$, we have $x_{k+m}^d = x_{m}^d.$ Consequently, setting $\y \coloneqq \mathcal{Q}(\x)$,
		\[
\left\{
\begin{array}{rc}
y_{k+1} = x_{k}^d +y_1 = x_{k+m}^d +y_1 = y_{k+m+1}& \\
y_{i} = x_{i-1}^d +y_1 = x_{m+i-1}^d +y_1 = y_{i+m} & \text{ for all $i \ge k+2$}.
\end{array}
\right.		\]
Thus $\mathcal{Q} (\x) \in \mathcal{P}_{k,m}$.

Since $\p_{k,m}(\mathcal{P}_{k,m}) \subset \mathcal{P}_{k,m}$, $F_{k,m} (\Mkm)  \subset \mathcal{P}_{k,m}$. Since $\Fkm(\mathcal{E}) \subset \mathcal{H}_{k,m}$,
	\[
\Fkm(\Mkm) \subset \mathcal{P}_{k,m} \cap \mathcal{H}_{k,m} =  \Mkm.\]
		Clearly, the map $\mathcal{Q}$ is homogeneous of degree $d$ and the map $\pi_{k,m}$ is homogeneous of degree $1$, thus $F_{k,m}$ is homogeneous of degree $d$. 
		
		Let us now prove that $\Fkm$ is nondegenerate, i.e. $\Fkm^{-1}(\boldsymbol{0})= \{ \boldsymbol{0} \}$. Assume $ \x \in \Mkm$ and \[{\p_{k,m} \circ \mathcal{Q}(\x) = \Fkm(\x) = 0}.\] Then $\boldsymbol{y} \coloneqq \mathcal{Q}(\x) \in \Ker(\pi_{k,m}) = \mathcal{L}.$ By definition of $\mathcal{Q}$, $y_1 =0$. Since $\y \in \mathcal{L}$, $\boldsymbol{y}$ is a constant sequence thus $\y = \textbf{0}$. This implies that $x_i = 0$ for all $i \ge 1$, i.e. $\x = \textbf{0}$. 
		
		Since $\Mkm$ has finite dimension and $\Fkm \colon \Mkm \to \Mkm$ is homogeneous and nondegenerate, $\Fkm$ is surjective, i.e. $F_{k,m}(\M_{k,m}) = \M_{k,m}$.
	\end{proof}

	\begin{lemma}\label{lem: largest invariant}
	We have that	$\bigcap\limits_{n \ge 1} \Fkm^{\circ n}(\mathcal{E}) = \Mkm.$
	\end{lemma}
	\begin{proof}
		According to the previous lemma, $\Fkm(\Mkm) = \Mkm$. Thus,
		\[
		\mathcal{M}_{k,m} \subseteq \bigcap\limits_{n \ge 1} \Fkm^{\circ n}(\mathcal{E}).
		\]
		Conversely, it is enough to prove that 
		\begin{equation}\label{eq: Mkm}
		\bigcap\limits_{n \ge 2} \Fkm^{\circ n}(\mathcal{E}) \subset \cp_{k,m}.		
		\end{equation}
Indeed, since $\Fkm (\mathcal{E}) \subset \mathcal{H}_{k,m}$ and $\mathcal{M}_{k,m} = \mathcal{H}_{k,m} \cap \mathcal{P}_{k,m}$, the inclusion \eqref{eq: Mkm} implies that
\[ \Fkm(\mathcal{E}) \cap \bigcap\limits_{n \ge 2} \Fkm^{\circ n}(\mathcal{E}) \subseteq \mathcal{H}_{k,m} \cap \mathcal{P}_{k,m},
\]
and hence
\[\bigcap\limits_{n \ge 1} \Fkm^{\circ n}(\mathcal{E}) \subseteq \Mkm \subseteq \bigcap\limits_{n \ge 1} \Fkm^{\circ n}(\mathcal{E}).\]

To prove \eqref{eq: Mkm}, we show the following claim: for all $n \ge 2$, if $\y \in \Fkm^{\circ n}(\mathcal{E})$, then $y_{i+m} = y_{i}$ for all $ i \in \{ k+1,\ldots,k+n-1\}$.
		 
		 Let us prove this claim by induction in $n$. If $n=2$, assume $\y \in \Fkm^{\circ 2}(\mathcal{E})$, i.e. $\y =\Fkm^{\circ 2}(\x)$ for some $\x \in \mathcal{E}$. Setting $\z = \Fkm(\x)$, according to Lemma \ref{lem:explicit form of Fkm}, we have $\beta z_{k+m} - z_{k} = 0$. Thus, since $\y = \Fkm(\z)$ and since $\beta^d = 1$, also according to Lemma \ref{lem:explicit form of Fkm}, we have
		 \[
		 y_{k+1+m} = z_{k+m}^d + y_1 = z_{k}^d + y_1 = y_{k+1},
		 \]
i.e. the claim is true for $n = 2.$

 Assume that it holds for some $n > 2$. Assume $\boldsymbol{y} \in \Fkm^{\circ(n+1)}(\mathcal{E}),$ i.e. $\boldsymbol{y} = \Fkm(\x)$ with $\x \in \Fkm^{\circ n}(\mathcal{E})$. The induction hypothesis implies that\[  x_{i+m-1} = x_{i-1} \text{ for all $i \in \{ k+2,\ldots,k+n \}$},\] so that \[y_{i+m} =y_i \text{ for all $ i \in \{k+2,\ldots,k+n\}$}.\] In addition, since $\x \in \mathcal{H}_{k,m}$, $\b x_{k+m} = x_k$, so that $y_{k+m+1} = y_{k+1}$. Thus the claim is true for $n+1$.
	\end{proof}

\subsection{The main result about $F_{k,m}$}
Lemma \ref{lem: moduli map is nondegenerate} and Lemma \ref{lem: largest invariant} allow us to restrict our study to the dynamics of $\Fkm$ on $\Mkm$. With a slight abuse of notations, from now on, we shall denote by $\Fkm$ the restriction of ${\Fkm\colon \mathcal{E} \to \mathcal{E}}$ to $\Mkm$. The following result sums up the properties of $F_{k,m}\colon\M_{k,m} \to \M_{k,m}$ which are important for us.
\begin{themprime}\label{thm: Koch prime}
Given $(k,m) \in \N \times \N^*$ and let $\z \in \Mkm \setminus \{ \boldsymbol{0} \}$ be a fixed point of the map $\Fkm\colon\Mkm \to \Mkm$. Let $k'$ be the preperiod of the sequence $\z$ and $m'$ be its period. Then,
\begin{enumerate}
	\item\label{itm:first} the polynomial $P(t) = t^d + z_1 \in \C[t]$ is post-critically finite and we have ${\z = (P^{\circ j}(0))_{j \ge 1}}$; in particular, the critical value $z_1$ of $P$ is preperiodic of preperiod $k'$ to a cycle of period $m'$ of multiplier $\lambda$,
	\item\label{itm:second} there exists a partial order $\preceq$ on $\N \times \N^*$ such that $(k',m') \preceq (k,m)$ if and only if $\M_{k',m'} \subseteq \Mkm$ is a $\Fkm$-invariant subspace and the restriction of $\Fkm$ to $\M_{k',m'}$ is $F_{k',m'}$,
	\item\label{itm:third} if $k+m-1 \ge 2$, $F_{k,m}\colon\M_{k,m} \to \M_{k,m}$ and $G_{k,m}\colon\C^{k+m-1} \to \C^{k+m-1}$ are holomorphically conjugate,
	\item\label{itm:fourth} $\Spec D_{\z} F_{k',m'} \subset \C \setminus \overline{\D}$,
	\item\label{itm:fifth} If $(k,m) \neq (k',m')$, \[\Spec \left(D_{\z} \Fkm\right)^*|_{ (\M_{k',m'})^0} = \left\{\begin{array}{ll}
	\{0\} & \text{ if $k' =0$}\\
	\{ \mu \mid \mu^m = \lambda^{\frac{m}{m'}}, \mu^{m'} \neq \lambda \} & \text{ if $k' \neq 0$}
	\end{array}\right..\]
	where $\M_{k',m'}^0 = \{\omega \in \Mkm^* \mid \omega|_{\M_{k',m'}} \equiv 0 \}$ is the annihilator of $ \M_{k',m'}$ in $\Mkm$
\end{enumerate}
\end{themprime}
In particular, we shall see that Theorem \ref{thm:theoremA'} is a direct consequence of Theorem \ref{thm: Koch prime}. The rest of this article is devoted to the proof of Theorem \ref{thm: Koch prime}. Item \ref{itm:first} is proved in Proposition \ref{prop: critical orbit}. The partial order $\preceq$ will be introduced in Definition \ref{def: partial order} and item \ref{itm:second} will be proved in Proposition \ref{prop: partial order}. Item \ref{itm:third} is proved in Lemma \ref{prop: FKm and GKm}. Item \ref{itm:fourth} is due to Koch, and we recall its proof in Proposition \ref{prop: fixed point outside} for the sake of completeness. Our main contribution is the proof of item \ref{itm:fifth} which will be proved in Proposition \ref{prop: eigenvalue inside}. Finally, we prove Theorem \ref{thm:theoremA'} by using Theorem \ref{thm: Koch prime}.
\section{Dynamics of $\Fkm$}

\subsection{$F_{k,m}$ is conjugate to $G_{k,m}$}\label{sec: conjugate to Sarah maps}
The following lemma assures that the class of maps $\Fkm$ for $ (k,m) \in \N \times \N^*$ is a good alternative when one wants to study $G_{k,m}$. 
\begin{lemma}\label{prop: FKm and GKm}
Let $(k,m) \in \N \times \N^*$ be such that $k+m-1 \ge 2$, the maps $\Fkm$ and $G_{k,m}$ are holomorphically conjugate.
\end{lemma}
\begin{proof} Recall that when $k=0, m \ge 3$, we have
	\[
	G_{0,m}: \left(
	\begin{array}{c}
	x_1\\
	x_2\\
	\vdots\\
	x_{m-1}
	\end{array}
	\right) \mapsto \left(
	\begin{array}{c}
	- x_{m-1}^d\\
	x_1^d - x_{m-1}^d\\
	\vdots\\
	x_{m-2}^d - x_{m-1}^d
	\end{array}
	\right).
	\]
	and when $k \ge 1, m \ge 2$, we have
	\[
	G_{k,m}: \left(
	\begin{array}{c}
	x_1\\
	x_2\\
	\vdots\\
	x_{k+m-1}
	\end{array}
	\right) \mapsto \left(
	\begin{array}{c}
	\left( - \frac{\b x_{k+m-1} - x_{k-1}}{\b-1} \right)^d\\
	\left(x_1 - \frac{\b x_{k+m-1} - x_{k-1}}{\b-1} \right)^d\\\\
	\vdots\\
	\left( x_{k+m-2} - \frac{\b x_{k+m-1} - x_{k-1}}{\b-1} \right)^d
	\end{array}
	\right).
	\]
Let $i_{k,m} \colon \Mkm \to \C^{k+m-1}, i_{k,m}(\x) = (x_1,\ldots,x_{k+m-1})$ and set \[\widetilde{\Fkm} = i_{k,m} \circ \Fkm \circ i_{k,m}^{-1}.\] Then we have
 when $k = 0$, 
\[
\widetilde{F_{0,m}}: \left(
\begin{array}{c}
x_1\\
x_2\\
\vdots\\
x_{m-1}
\end{array}
\right) \mapsto \left(
\begin{array}{c}
- x_{m-1}^d\\
x_1^d - x_{m-1}^d\\
\vdots\\
x_{m-2}^d - x_{m-1}^d
\end{array}
\right).
\]
and when $k \ge 1$,
\[
\widetilde{F_{k,m}}: \left(
\begin{array}{c}
x_1\\
x_2\\
\vdots\\
x_{k+m-1}
\end{array}
\right) \mapsto \left(
\begin{array}{c}
- \frac{\b x_{k+m-1}^d - x_{k-1}^d}{\b-1} \\x_1^d - \frac{\b x_{k+m-1}^d - x_{k-1}^d}{\b-1}\\
\vdots\\
x_{k+m-2}^d - \frac{\b x_{k+m-1}^d - x_{k-1}^d}{\b-1} 
\end{array}
\right).
\]
It is enough to show that $G_{k,m}$ and $\widetilde{\Fkm}$ are conjugate. Indeed, let $\tau \colon \C^{k+m-1} \to \C^{k+m-1}$ be a linear map of the following form
\[
\tau\left(
\begin{array}{c}
x_1\\
x_2\\
\vdots\\
x_{k+m-1}
\end{array}
\right) \mapsto \left(
\begin{array}{c}
\tau_1\\
x_1 + \tau_1\\
\vdots\\
x_{k+m-2}+\tau_1
\end{array}
\right) \text{ with } \tau_1 = \begin{cases}
-x_{m-1} \text{ when $k = 0$}\\
- \frac{\b x_{k+m-1} - x_{k-1}}{\b-1} \text{ when $k \ge 1$}
\end{cases}
\]	
and set 
\[
\mathfrak{d}(x_1,\ldots,x_{k+m-1}) = (x_1^d,\ldots,x_{k+m-1}^d)
\]
then $G_{k,m} = \mathfrak{d} \circ \tau, \widetilde{\Fkm} = \tau \circ \mathfrak{d}$. Thus,
\[
\tau \circ G_{k,m} = \widetilde{\Fkm } \circ \tau
\]
Note that $\tau$ is an isomorphism, whence $G_{k,m}$ and $\widetilde{\Fkm}$ are conjugate.

	\end{proof}
\subsection{Comparing $\Fkm$ by a partial order $\preceq$}
Our initial expectation was that for arbitrary pairs $(k_1,m_1)$ and $(k_2,m_2)$ in $\N \times \N^*$, the maps $F_{k_1,m_1}$ and $F_{k_2,m_2}$ would agree on the intersection $\mathcal{M}_{k_1,m_1} \cap \mathcal{M}_{k_2,m_2}$. However this is not true as shown in the following example. Consider the case $ d = 2$ and $\beta = -1$, the sequence
\[
\x\coloneqq\{2,0,0,\ldots,0,\ldots\}
\]
Then $\x \in \M_{2,1} \cap \M_{3,1}$ and 
\[
\mathcal{Q}(\x)= \{0,4,0,0,0,0,0,\ldots\}
\]
However,
\[
F_{2,1}(\x)=\{-2,2,-2,-2,-2,-2,\ldots\} \text{ and } F_{3,1}(\x) = \{0,4,0,0,0,0,0,\ldots\}.
\]
We will now see that if some order $(k',m') \preceq (k,m)$ is satisfied, with the fixed $d$ and $\b$, we have $\M_{k',m'} \subseteq \Mkm$ and $F_{k',m'}$ is the restriction of $\Fkm $ to $\M_{k',m'}$.
\begin{definition}\label{def: partial order}
	Let $\preceq$ be the partial order on $\N \times \N^*$ defined by \[
	(k',m') \preceq (k,m) \Leftrightarrow \left\{
	\begin{array}{l}
	m' \text{ divides } m,\\
	\text{either } k'=k \text{ or } (k'=0 \text{ and } m' \text{ divides } k).
	\end{array}
	\right.
	\]
	The strict order $\prec$ is defined by \[(k',m') \prec (k,m) \Leftrightarrow (k',m') \preceq (k,m) \text{ and } (k',m') \neq (k,m). \]
\end{definition}
\begin{proposition}\label{prop: partial order} For two pairs of integers $(k,m),(k',m') \in \N \times \N^*$, we have that
	\[
	(k',m') \preceq (k,m) \Leftrightarrow
	\mathcal{M}_{k',m'} \subseteq \mathcal{M}_{k,m}
	\]
	Moreover, if $\M_{k',m'} \subseteq \Mkm$ then $F_{k,m}|_{\mathcal{M}_{k',m'}} = F_{k',m'}.$
\end{proposition}
\begin{proof} We first prove that
	\[
	(k,m) \preceq (k',m') \Leftrightarrow \M_{k',m'} \subseteq \Mkm.
	\] Assume that $(k',m') \preceq  (k,m)$. We shall prove that $\M_{k',m'} \subseteq \Mkm$. Assume $\x \in \M_{k',m'}$. Then $\x$ is preperiodic of period less than $k'$ to a cycle of period dividing $m'$. Since $k' \le k$ and $m' \mid m$, we deduce that $\x \in \mathcal{P}_{k,m}$. We need to show that $\b x_{k+m} -x_k =0$. Indeed,
	\begin{itemize}
		\item if $k'=k$, since $m' \mid m$ and $\x \in \M_{k',m'}$, we have $x_{k+m}=x_{k'+m'}$. Thus $\b x_{k+m} - x_k = \b x_{k'+ m'} - x_{k'} = 0$,
		\item if $k'=0$, in that case $m' \mid k$ and $\x$ is periodic of period $m'$. Thus $x_{k+m} =  x_{k} =0$.
	\end{itemize}
Let us now assume that $\M_{k',m'} \subseteq \Mkm$. We claim that $(k',m') \preceq (k,m)$. Indeed,
\begin{itemize}
	\item either $k'=0$; in this case, consider $\x \in \M_{0,m'}$ given by $x_i = 0 $ if $m' \mid i$ and $1$ otherwise. Since $\x \in \Mkm$, $\b x_{k+m} - x_{k} = 0$ with $\b \neq 1$. Then necessarily, $x_{k+m} = x_{k} = 0$ thus $m'$ divides $k$ and $m$.
\item or $k' \ge 1$; in this case, consider $\x \in \M_{k',m'}$ given by $x_{k'} = \b, x_{k'+jm'} = 1$ for $j \ge 1$ and $x_{i} = 2$ otherwise. If $\b x_{k+m}-x_k = 0$ then $x_k = \b $ and $x_{k+m} = 1$. Hence $k = k'$ and $m' \mid m$.
\end{itemize}

We assume now $\M_{k',m'} \subseteq \Mkm$, or equivalently, $(k',m') \preceq (k,m)$. Let us prove that the restriction of $\Fkm$ to ${\M_{k',m'}}$ is $ F_{k',m'}$. Assume $\x \in \M_{k',m'}$. Set $\y = F_{k',m'}(x)$ and $\z= \Fkm (\x)$. According to Lemma \ref{lem:explicit form of Fkm}, for all $i \ge 2$, $y_i = x_{i-1}^d +y_1,z_i = x_{i-1}^d + z_1$ where 
\[	y_1= \left\{\begin{array}{lc}
	-x_{m'-1}^d & \text{ if } k' = 0\\
	-\frac{\b x_{k'+m'-1}^d -x_{k'-1}^d}{\b-1} & \text{ if } k' \ge 1
	\end{array} \right.
\quad  \text{ and } \quad z_1= \left\{\begin{array}{lc}
	-x_{m-1}^d & \text{ if } k = 0\\
	-\frac{\b x_{k+m-1}^d - x_{k'-1}^d}{\b-1} & \text{ if } k \ge 1.
	\end{array} \right.
\]
It is enough to prove that $y_1 = z_1$.

\noindent$\bullet$ Case $k'=0$. In that case, $k$ and $m$ are multiples of $m'$. If $k = 0$ then \[y_1 = -x_{m'-1}^d = -x_{m-1}^d = z_1.\] If $k \neq 0$, $
x_{k+m-1} = x_{k-1} = x_{m'-1}$ hence \[y_1 = -x_{m'-1}^d  = - \frac{\b x_{k+m-1}^d - x_{k-1}^d}{\b-1} =z_1.\]

\noindent$\bullet$ Case $k' \neq 0$. In that case, $k' =k$ and $m$ is a multiple of $m'$. Then 
\[
x_{k+m-1} = \left\{
\begin{array}{cl}
\frac{1}{\b}x_{k'+m'-1} & \text{ if $m'-1=0$ and $m-1 \ge 1$} \\
x_{k'+m'-1} &\text{ otherwise.}
\end{array}
\right.
\]\footnote{Note that if $m' = 1$, $ \x \in \M_{k',m'}$ implies that for all $i \ge 1$, $x_{k'+i} = x_{k'+1}$. In particular, with $k = k'$, $x_{k+m-1} = x_{k'+1} = \frac{1}{\b} x_{k'}$ }
Then 
\[y_1 = - \frac{\b x_{k'+m'-1}^d-x_{k'-1}^d }{\b-1}  = - \frac{\b x_{k+m-1}^d -x_{k-1}^d}{\b-1} =z_1. \qedhere\]
\end{proof}

\subsection{The post-critical set of $\Fkm$}
In this section, we fix a pair of integers $(k,m) \in \N\times \N^*$. Recall that \[
C(F_{k,m})\coloneqq \text{the critical set of }  F_{k,m} \colon \M_{k,m} \to \M_{k,m},\]\[CV(F_{k,m})\coloneqq  \text{ the critical value set of } F_{k,m}\colon\M_{k,m} \to \M_{k,m} 
\]
and 
\[
PC(F_{k,m}) \coloneqq \text{ the post-critical set of } F_{k,m} \colon \M_{k,m} \to \M_{k,m}.
\]
%
%

\begin{lemma}\label{lem: critical of Fkm}
We have that $C(F_{k,m})= \{ \x \in \M_{k,m} \mid x_i = 0 \text{ for some $ 1 \le i \le k +m-1$} \}$
\end{lemma}
\begin{proof}
	Recall that $F_{k,m} = \p_{k,m} \circ \mathcal{Q}$. Differentiating both sides, we see that for any $\x \in \M_{k,m}$ and for any $\boldsymbol{v} \in T_{\x} \mathcal{E} = \mathcal{E}$,
	\[
	D_{\x} F_{k,m} (\boldsymbol{v}) = \pi_{k,m} \circ D_{\x} \mathcal{Q} (\boldsymbol{v}) = \pi_{k,m} (0,dx_1^{d-1}v_1,dx_2^{d-1}v_2,\ldots) 
	\]
On the one hand, assume $\x \in C(F_{k,m})$. Then there exists $\boldsymbol{v} \in T_{\x} \M_{k,m} \setminus \{\textbf{0}\}$ such that  $x_i^{d-1}v_i = 0$ for all $i \ge 1$. Observe that there exists $i \in  \{1,\ldots, k+m-1\} $ such that $ v_i \neq 0$ whence $x_i = 0$. Indeed otherwise, $v_{k+m} = \frac{1}{\b} v_{k} =0$ and by preperiodicity, $v_i =0$ for all $i \ge 1$. 

On the other hand, given $\x \in \mathcal{E}$, define $\boldsymbol{v} \in \mathcal{E}$ by 
\[
v_j = \left\{
\begin{array}{cl}
\b & \text{ if $x_j = 0$ and $j= k$}\\
1 & \text{ if $x_j = 0$ and $j \neq k$}\\
0 & \text{ if $x_j \neq 0$}
\end{array}
\right.
\]
Then $x_j^{d-1} v_j= 0$ for all $j \ge 1$ so that $D_{\x} F_{k,m} (\boldsymbol{v}) = \textbf{0}$. Moreover, if $\x \in \M_{k,m}$ then $\boldsymbol{v} \in \M_{k,m}$. Finally, if there exists $i \in \{ 1,\ldots, k+m-1 \}$ such that $x_i = 0$ then $\boldsymbol{v} \neq \textbf{0}$ whence $\x \in C(F_{k,m})$. 
\end{proof}

\begin{definition}
Denote by \[
\Delta_{k,m} = \{
\x \in \M_{k,m} \mid \text{ there exists $1 \le i < j\le k+m$ such that $x_i = x_j$}
\}
\]
\end{definition}
The set $\Delta_{k,m}$ consists of $\binom{k+m}{2}$ hyperplanes.

\begin{proposition}\label{prop: critical values} We have that $CV(F_{k,m}) \subseteq \Delta_{k,m}$ and $F_{k,m}(\Delta_{k,m}) \subseteq \Delta_{k,m}$. Consequently, ${PC(F_{k,m}) \subseteq \Delta_{k,m}}$.
\end{proposition}
\begin{proof}

Let $ \x \in C(F_{k,m})$ and set $\y = F_{k,m}(\x)$. Then by Lemma \ref{lem: critical of Fkm}, there exists $i \in  \{1,\ldots , k+m-1\} $ such that $x_i = 0$. By Lemma \ref{lem:explicit form of Fkm}, we have
\[
y_{i+1} = x_{i}^d + y_1 = y_1
\]
Thus $\y \in \Delta_{k,m}$, whence $CV(F_{k,m}) \subseteq \Delta_{k,m}$.

Now we prove that $\Delta_{k,m}$ is invariant under $F_{k,m}$. Assume $\x \in \Delta_{k,m}$ and set $\y = F_{k,m}(\x)$. Then there exist $1 \le i < j \le k+m$ such that $x_i = x_j$. By Lemma \ref{lem:explicit form of Fkm}, for every $l \ge 2, y_l = x_{l-1}^d +y_1$. Note that since $\x\in \M_{k,m}$, $\b x_{k+m} -x_k  =0$ with the convention $x_0 \coloneqq0$.
\begin{itemize}
	\item If $j \le k+m-1$, we have $y_{i+1} = x_{i}^d + y_1 = x_{j}^d +y_1 = y_{j+1}$.
	\item If $j = k+m$, then
	\begin{itemize}
	\item either $i = k $ so that $x_k = x_ i = x_j = x_{k+m}$; since $\b  x_{k+m} -x_k = 0$ and $\b \neq 1$, $x_k =0$ whence $\x \in C(F_{k,m})$ and $\y \in \Delta_{k,m}$;
	\item or $i \neq k$ so that $i+1 \neq k+1$; since $x_{k} = \b x_{k+m} = \b x_{i},$ we have
	\[
	y_{i+1} = x_{i}^d + y_1 = x_{k}^d + y_1 = y_{k+1}.
	\]
	\end{itemize}

\end{itemize}
Hence, in any case, we have $\y \in \Delta_{k,m}$, i.e. $F_{k,m}(\Delta_{k,m}) \subset \Delta_{k,m}$ and the lemma is proved. 
\end{proof}

\section{Fixed points of Koch maps}
In this section, we shall study the eigenvalues of the derivative of $F_{k,m}$ at its fixed points and we will prove Theorem \ref{thm: Koch prime}. Then, we deduce Theorem \ref{thm:theoremA'} by using Theorem \ref{thm: Koch prime}.
\subsection{Relation with post-critically finite polynomials}
There is a close connection between fixed points $\Fkm$ and post-critically finite polynomials. More precisely, we will consider monic centered unicritical polynomials of degree $d \ge 2$,\[P(t) = t^d +c \in \C[t], c\in \C\]
The critical orbit of such a polynomial is the sequence $\boldsymbol{c}_P \in \mathcal{E}$ defined by 
\[
\boldsymbol{c}_P= (c_i)_{i \ge 1} \in \mathcal{E} \text{ where } c_i = P^{\circ i}(0).
\] Since the preperiod and the period of a preperiodic sequence will be extensively discussed in this chapter, we introduce the following notions. 
\begin{definition}
	Given integers $k \ge 0, m \ge 1$, a sequence $\x \in \mathcal{E}$ is called \textit{preperiodic of type $(k,m)$} if for every $i \ge k+1$, $x_{i+m} = x_i$, \textit{preperiodic of \textbf{exact} type $(k,m)$} if, additionally, $k$ and $m$ are the smallest integers satisfying such conditions. 
\end{definition}
For a sequence of exact type $(k,m)$, the pair $(k,m)$ consists of the preperiod $k$ and the period $m$. The vector space $\mathcal{P}_{k,m}$ is the space of preperiodic sequences of type $(k,m)$.
\begin{definition}
	A {\em degree $d$ polynomial of (exact) type $(k,m)$} is a monic centered unicritical polynomial $P$ of degree $d \ge 2$ whose critical orbit ${\boldsymbol c}_P$ is of (exact) type $(k,m)$. 
\end{definition}
In other words, a polynomial is of type $(k,m)$ if and only if its critical orbit belongs to $\mathcal{P}_{k,m}$. Note that a polynomial of type $(k,m)$ is post-critically finite. 
\begin{remark}\label{rm: type of pcf}
	Let $P$ be a polynomial of type $(k,m)$ of degree $d$. If $k=0$, then the critical value $c$ of $P$ is a periodic point of period dividing $m$, i.e. $P^{\circ m}(c) = c$. In other words, $P^{\circ (m-1)}(c) \in P^{-1}(c)$. However, since $P$ is a unicritical polynomial, $P^{-1}(c)$ consists of exactly one point which is the critical point of $P$. This means that $P^{\circ (m-1)}(c)$ is in fact the critical point of $P$. This is the case if and only if the critical point of $P$ is also a periodic point of type $(0,m)$.
\end{remark}
\begin{proposition}\label{prop: critical orbit}Given $(k,m) \in \N \times \N^*$. Let $\z\in \M_{k,m}$ be a fixed point of $F_{k,m}$. Set $P(t) = t^d +z_1$. Then $\z = \boldsymbol{c}_P$ and $P$ is of exact type $(k',m')$. Moreover, $(k',m') \preceq (k,m)$.
\end{proposition}
\begin{proof}

First, let us prove that $P$ is of type $(k,m)$. According to Lemma \ref{lem:explicit form of Fkm}, for every $i \ge 2$, we have that $z_i = z_{i-1}^d+z_1$ hence
\[
z_i = P(z_{i-1}).
\] In other words, $\z$ is the sequence of iterates of $z_1$ under $P$. Recall that by Lemma \ref{lem: largest invariant}, $\M_{k,m}$ contains every fixed point of $F_{k,m}$ hence $\z \in \M_{k,m} \subset \mathcal{P}_{k,m}$. Therefore, $\z = \boldsymbol{c}_P$ and the polynomial $P$ is a polynomial of type $(k,m)$. 

Second, let $(k',m')$ be the exact type of $P$. We prove that $(k',m') \preceq (k,m)$, i.e.
\[
\left\{
\begin{array}{l}
m' \mid m\\
\text{either } k' = k \text{ or } (k' = 0 \text{ and } m' \mid k)
\end{array}
\right.
\] 
Since $(k',m')$ is the exact type of the orbit of $z_1$, $k' \le k$ and $m' \mid m$. If $k' = k$, we are done. If $k' \neq k$, we need to prove that $k'=0$ and $m' \mid k$. Since $(k',m')$ is the exact type of $\z$, we have $\z \in \mathcal{P}_{k',m'}$. Thus, $k'+1 \le k$ and $m' \mid m$; and since $\z \in \mathcal{P}_{k',m'}$, this implies that $z_{k+m} =z_k$. Moreover, $\z \in \M_{k,m}$ implies that $\b z_{k+m} - z_{k}=0$. Therefore, $P^{\circ k}(0) =z_k =0.$ In other words, $0$ is a periodic point of $P$, i.e. $k'=0$, and the period of $0$ is $m'$. Moreover, $P^{\circ k}(0)  = 0$ also implies that $k$ is a multiple of $m'$. Thus, we can conclude that $(k',m') \preceq (k,m)$. 
\end{proof}

\begin{remark}
The converse statement of Proposition \ref{prop: critical orbit} is true under some assumptions on the choice of the root of unity $\beta$. More precisely, given a post-critically finite unicritical polynomial $P(t) = t^d + z_1$ of type $(k,m)$, there exists a $d$-th root of unity $\beta' \neq 1$ such that the critical orbit $\boldsymbol{c}_P$ of $P$ is a fixed point of $F_{k,m}$. 
\end{remark}

The partial order $\preceq$ enables us to study the relative positions of the fixed points of $F_{k,m}$ and $\Delta_{k,m}$.
\begin{lemma}\label{lem: position of fixed points}
Let $\z$ be a fixed point of $F_{k,m}$ and let $(k',m')$ be the exact type of $\z$. Then, $\z \in \Delta_{k,m}$ if and only if $(k',m') \prec (k,m)$.
\end{lemma}
\begin{proof} 

Assume $\z \in \Delta_{k,m}$, i.e. there exists $1\le i<j \le k+m$ such that $z_i = z_j$. In particular, $\z$ is a preperiodic sequence of preperiod at most $i-1$ and of period dividing $j-i$. Whence, since $(k',m')$ is the exact type of $\z$, we have $k' \le i-1$ and $m'$ divides $j-i$. 
\begin{itemize}
	\item If $i \le k$ then $k' \le i-1 < k$.
	\item If $i \ge k+1$ then $j-i \le k+m -(k+1) < m$. Since $m' \mid  j-i$, we have $m' < m$.
\end{itemize}
In both cases, we have $(k',m') \neq (k,m)$.  Note that, according to Proposition \ref{prop: critical orbit}, $(k',m) \preceq (k,m)$. Hence $(k',m') \prec (k,m)$.

Conversely, assume $(k',m') \prec (k,m)$. In particular, $k' \le k$, $m' \le m$ and $(k',m') \neq (k,m)$. Note that $\z$ is of exact type $(k',m')$. Hence, \[z_{k'+1} = z_{k'+m'+1}.\]If $k' \neq k$ then $k' < k$. Whence $k'+1$ and $k'+m'+1$ are integers in $ \{1, \ldots , k+m \}$. If $k' = k$ then $m' < m$. In this case, $k+1,k+m'+1$ are also in $\{ 1,\ldots,k+m\} $. Therefore, in both cases, we deduce by that $\z \in \Delta_{k,m}$.
\end{proof}
\subsection{Eigenvalues of moduli maps at fixed points}
In order to study the eigenvalues of the derivative of moduli maps at one of its fixed point, we will in fact study its transpose. Note that when $k+m =1$, $\M_{k,m} = \{\textbf{0}\}$ and $F_{k,m}$ is trivial. Let us fix $(k,m) \in \N \times \N^*$ such that $k+m \ge 2$. Assume that $\z \in \M_{k,m}$ is a fixed point of $F_{k,m}$. We will describe the transpose of the derivative $D_{\z} F_{k,m}\colon T_{\z}\M_{k,m} \to T_{\z}\M_{k,m}$. Since $\M_{k,m}$ is a vector space, there is a canonical identification of $T_{\z} \M_{k,m}$ with $\M_{k,m}$,  the derivative $D_{\z} F_{k,m}\colon T_{\z} \M_{k,m}\to T_{\z} \M_{k,m}$ identifies with a linear map \[L:\M_{k,m} \to \M_{k,m},\] 
and the transpose identifies with the pull-back map of $L$
\[
L^*\colon \M_{k,m}^* \to \M_{k,m}^*.
\]
\subsubsection{The dual space $\M_{k,m}^*$}
For $i \ge 1$, let $\omega_i \in \M_{k,m}^*$ be the linear form defined by for all $\boldsymbol{v} \in \M_{k,m}$,
\[
 \omega_i(\boldsymbol{v}) \coloneqq  v_i.
\]
\begin{lemma}\label{lem: basis of E^*}The family $\{ 
	\omega_i, 1\le i \le k+m-1	\}$ is a basis of $\M_{k,m}^*$.
\end{lemma}
\begin{proof}Note that $\dim \M_{k,m} = k+m-1$ hence it is enough to prove that $\{\omega_1,\ldots,\omega_{k+m-1} \}$ are linearly independent. Assume that
	\[
\sum\limits_{1\le i \le k+m-1}\lambda_ i \omega_ i = 0 \text{ with $\lambda_i \in \C$}.
	\]
Let $ i \ge 1$.  To prove that $\lambda_i = 0$, consider the vector $\boldsymbol{v} \in \M_{k,m}$ defined by
\begin{itemize}
	\item if $i< k$, $v_j = \begin{cases}1 &\text{ if $j=i$}\\
	0& \text{ otherwise},
	\end{cases}$
	\item if $i = k$, $v_j = \begin{cases}
	1 &\text{ if $j = i =k$}\\
	\frac{1}{\b}&\text{ if $j > k$ and $j \equiv k \mod{m}$}\\
	0 & \text{ otherwise},
	\end{cases}$
	\item if $i > k$, $v_j = \begin{cases}
	1 &\text{ if $j \ge i$ and $j\equiv i \mod{m}$}\\
	0& \text{ otherwise.}
	\end{cases}$
\end{itemize}
In any case, we have

\[0 = \sum\limits_{1\le i \le k+m-1}\lambda_ i \omega_ i (\boldsymbol{v}) = \lambda_i v_i = \lambda_i.\qedhere
	\]
\end{proof}
\subsubsection{The transpose of the derivative}

\noindent Observe that $L^*\colon \M_{k,m}^*\to \M_{k,m}^*$ is the pull-back of forms, i.e. for all $\omega \in \M_{k,m}^*$,  \[L^*\omega = \omega \circ L.\]
For all $i \ge 1$, set
\[
\delta_{i} = dz_{i}^{d-1}\]  
where $\z = (z_1,z_2,\ldots) \in \M_{k,m}$ is the considered fixed point of $F_{k,m}$.
\begin{lemma}\label{lem: pull-back of L}
	We have that 
		\[
		L^*\omega_1 = \left\{
		\begin{array}{ll}
		-\delta_{m-1} \omega_{m-1} &\text{ if $k = 0$}\\
		-\frac{\b \delta_{k+m-1} \omega_{k+m-1} -\delta_{k-1} \omega_{k-1}}{\b-1} &\text{ otherwise,}
		\end{array}
		\right.
		\]
and for all $i \ge 2$, \[
		L^*\omega_i = \delta_{i-1} \omega_{i-1} +L^*\omega_1.
		\]
\end{lemma}
\begin{proof}
	Recall that for all $\boldsymbol{v} \in \M_{k,m}$, \[
	\begin{array}{rcl}
	L(\boldsymbol{v}) &=& \pi_{k,m} \circ D_{\z} \mathcal{Q}(\boldsymbol{v}).\\
	\end{array}
	\]
	Set $\boldsymbol{u} = D_{\z} \mathcal{Q} (\boldsymbol{v})$ then 
	\[
	u_1 = 0 \quad \text{ and for all $i \ge 2$}, \quad u_i = dz_{i-1}^{d-1} v_{i-1} = \delta_{i-1} v_{i-1}.
	\]
In addition, if $\boldsymbol{w} \coloneqq L(\boldsymbol{v})=\p_{k,m} (\boldsymbol{u})$ then 
	\[
	\text{ for all $i \ge 2$, } \quad  w_i = u_i +w_1 \text{ with } w_1 = \left\{
	\begin{array}{lc}
	-u_{m} &\text{ if } k =0\\
-\frac{\b u_{k+m-1} -u_{k-1}}{\b-1} &\text{ otherwise,}
	\end{array}
	\right.
	\]
	Combining those formulas, we obtain that $\boldsymbol{w} = L(\boldsymbol{v})$ and for all $i \ge 2$,
	\begin{equation}\label{eq: derivative Fkm}
	 w_{i} = \delta_{i-1} v_{i-1} +w_1 \text{ with } w_1 = \left\{
	\begin{array}{ll}
	-\delta_{m-1} v_{m-1} &\text{ if $k = 0$}\\
-\frac{\b \delta_{k+m-1} v_{k+m-1} -\delta_{k-1} v_{k-1}}{\b-1} &\text{ otherwise,}
	\end{array}
	\right.
	\end{equation}
We deduce that for all $\boldsymbol{v} \in \M_{k,m}$, 
	\[
	L^* \omega_1 (\boldsymbol{v}) = \omega_1\circ L(\boldsymbol{v}) =w_1= \left\{
	\begin{array}{ll}
	-\delta_{m-1} v_{m-1} &\text{ if $k = 0$}\\
-\frac{\b \delta_{k+m-1} \omega_{k+m-1} -\delta_{k-1} \omega_{k-1}}{\b-1} &\text{ otherwise,}
	\end{array}
	\right.
	\]
	hence
	\[
	L^*\omega_1 = \left\{
	\begin{array}{ll}
	-\delta_{m-1} \omega_{m-1} &\text{ if $k = 0$}\\
	-\frac{\b \delta_{k+m-1} \omega_{k+m-1} -\delta_{k-1} \omega_{k-1}}{\b-1} &\text{ otherwise.}
	\end{array}
	\right.
	\]
In addition, for all $i \ge 2$ and for all $\boldsymbol{v} \in \M_{k,m}$, we have
	\[
	L^*\omega_i(\boldsymbol{v}) = \omega_i\circ L(\boldsymbol{v}) = \delta_{i-1} v_{i-1} + w_1 =\delta_{i-1} v_{i-1} + \omega_1\circ L(\boldsymbol{v}),
	\]
	hence 
	\[L^*\omega_i = \delta_{i-1} \omega_{i-1} + L^*\omega_1.\qedhere\]
\end{proof}

\subsubsection{Fixed points outside the post-critical set}

According to Section \ref{sec: conjugate to Sarah maps}, the map $F_{k,m}$ is conjugate to the map $G_{k,m}$ constructed by Koch \cite{koch2013teichmuller}. By \cite[Corollary 7.2]{koch2013teichmuller}, the derivative of $G_{k,m}$ at its fixed points outside the post-critical set has only eigenvalues of modulus strictly greater than $1$, whence so does $F_{k,m}$. For the sake of completeness, we give here the proof of this property. For further discussion about the arithmetics of such eigenvalues, we refer to \cite{buff2017eigenvalues}. The main content of this paragraph is the following result.
\begin{proposition}\label{prop: fixed point outside} Let $(k,m) \in \N \times \N^*$ and $\z \notin PC(F_{k,m})$ be a fixed point of the moduli map $F_{k,m}$. Then every eigenvalue of $D_{\z} F_{k,m}$ has modulus strictly greater than $1$. 
\end{proposition}
\begin{proof} Since $\M_{k,m}$ has finite dimension, it is suffice to prove that every eigenvalue of the transpose $L^*$ of $D_{\z} F_{k,m}$ has modulus strictly bigger than $1$.  
	
Recall that, by Lemma \ref{lem: basis of E^*}, the family $\{\omega_{i}\colon \M_{k,m} \to \C\}_{i \in \{ 1,\ldots, k+m-1\}}$ is a basis of $\M_{k,m}^*$. According to Lemma \ref{lem: pull-back of L}, setting $\delta_i = dz_{i}^{d-1}$, we have
\[L^*\omega_1= \begin{cases} -\delta_{m-1}\omega_{m-1}&\text{if }k=0\\
\displaystyle -\frac{\b \delta_{k+m-1} \omega_{k+m-1} -\delta_{k-1}\omega_{k-1} }{\b -1}&\text{if }k\geq 1,\end{cases}\]
and for all $i\geq 2$, 
\[L^* \omega_{i} = \delta_{i-1} \omega_{i-1} + L^*\omega_1.\]

According to Lemma \ref{lem: position of fixed points}, the point $\z$ is a fixed point of $F_{k,m}$ of exact type $(k,m)$. Therefore, $\delta_{i} \neq 0$ for all $i\in  \{1,\ldots,k+m-1\}$. Indeed, note that according to Proposition \ref{prop: critical orbit}, the sequence $\z$ is the critical orbit of $P(t) = t^d + z_1$. Assume that $d z_i^{d-1} = \delta_{i}  = 0$ for some $i \in  \{1,\ldots, k+m-1\} $. Then $P^{\circ i}(0) = z_{i} = 0$. This implies that $k=0$ and $m$ divides $i$. However $i \le k+m-1 = m-1 < m$, hence contradiction.

We may therefore define a linear map 
$L_*:\M_{k,m}^*\to \M_{k,m}^*$ by 
\begin{equation}\label{eq:L}
\forall i\in \{ 1,\ldots, k+m-1\} \quad L_* (\omega_i) = \frac{\omega_{i+1} - \omega_1}{\delta_i}.
\end{equation}

\begin{lemma}
	The linear map $L^*$ is invertible and its inverse is $L_*$. 
\end{lemma}

\begin{proof}
	We need to prove that  $L_*\circ L^* = L^*\circ L_*  = {\rm id}$. 
	First, observe that for all $i\in \{1,\ldots,k+m-1\}$,
	\[
	\begin{array}{rcl}
L^*\circ L_*(\omega_i) = L^*\left(\frac{\omega_{i+1} - \omega_1}{\delta_i}\right) &=& \frac{1}{\delta_i}\left(L^*(\omega_{i+1}) - L^*(\omega_1)\right)\\
&=& \frac{1}{\delta_i}\bigl(\delta_i\omega_i + L^*(\omega_1)- L^*(\omega_1)\bigr) = \omega_i.
	\end{array}
	\]
	Second, we prove $L_* \circ L^* = \Id_{\M_{k,m}^*}$. Note that, by the definition of $\M_{k,m}$, we have that
	\[\omega_{k+m}  = \begin{cases}0&\text{if }k=0\\
	\frac{1}{\b}\omega_k &\text{if }k\geq 1\end{cases}\quad \text{and}\quad \forall i\geq k+m+1 \quad \omega_{i+m} = \omega_i.\] To compute $L_* \circ L^* (\omega_1)$, observe that if $k=0$, then 
	\[L_*\circ L^*(\omega_1) = L_*(-\delta_{m-1}\omega_{m-1}) = -\delta_{m-1} L_*(\omega_{m-1})= -(\omega_m-\omega_1) = \omega_1\]
	and if $k\geq 1$, then 
	\begin{eqnarray*}
		L_*\circ L^*(\omega_1) &=&L_*\left(-\frac{\b \delta_{k+m-1}  \omega_{k+m-1} - \delta_{k-1}\omega_{k-1}}{\b-1}\right) \\
		&=& -\left( \frac{\b}{\b-1}(\om_{k+m} - \om_1) - \frac{1}{\b-1}(\om_k - \om_1) \right) =\om_1.
	\end{eqnarray*}
	In both cases, $L_*\circ L^*(\omega_1)  = \omega_1$. 
	For $L_* \circ L^* (\omega_i)$ with $i\in  \{2,\ldots,k+m-1\}$, 
	\[L_*\circ L^*(\omega_i) = L_*\bigl(\delta_{i-1} \omega_{i-1} + L^*(\omega_1)\bigr) = \delta_{i-1}\frac{\omega_i-\omega_1}{\delta_{i-1}} + L_* L^*(\omega_1) = \omega_i.\]
	Thus, the linear map $L_*\colon \M_{k,m}^* \to \M_{k,m}^*$ is indeed the inverse of $L^*$
\end{proof}
In order to prove Proposition \ref{prop: fixed point outside}, it is therefore enough to prove that every eigenvalue of $L_*\colon \M_{k,m}^* \to \M_{k,m}^*$ is contained in the open unit disc $\D$. Inspired by the proof of \cite[Corollary 7.2]{koch2013teichmuller}, we will show that $L_*$ is conjugate to a linear transformation on a space of meromorphic quadratic differentials on $\C$, whose eigenvalues are all contained in $\D$. 

Consider the quadratic polynomial $P(t) \coloneqq t^d+z_1$, so that $z_i = P^{\circ i}(0)$ for all $i\geq 1$. 
Following Milnor \cite{milnor2014tsujii}, denote by $\mathfrak{Q}(\C)$ the space of meromorphic quadratic differentials on $\C$ which have at worst simple poles and let us use the notation $Q\in \mathfrak{Q}(\C)$ with 
\[Q = q(t) \text{d} t^2.\]
and $q(t)$ is a meromorphic function. Let $U\subset \C$ be a sufficiently large disk so that $P^{-1}(U)$ is compactly contained in $U$ and for $Q\in \mathfrak{Q}(\C)$, consider the norm 
\[\|Q\|_U \coloneqq \iint_U \bigl|q(t)\text{d} t^2\bigr|.\]
The pushforward of $Q$ by $P$ is the quadratic differential $P_*Q\in \mathfrak{Q}(\C)$ defined by 
\[P_*Q \coloneqq \sum_{P(u) = t} \frac{q(u)}{\bigl(P'(u)\bigr)^2}\text{d} t^2.\]
It follows from the triangle inequality that 
\[\|P_* Q\|_U \leq \|Q\|_{P^{-1}(U)} < \|Q\|_U.\]

For $i\geq 1$, let $Q_i\in \mathfrak{Q}(\C) $ be the quadratic differential defined by 
\[Q_i \coloneqq \frac{\text{d} t^2}{t-z_i}.\]
The following lemma generalizes a result due to Milnor, \cite[Lemma 1]{milnor2014tsujii} in the case $d=2$.
\begin{lemma}\label{lem: pull back}For all $ i\in \{1,\ldots,k+m-1\}$,
\begin{equation}\label{eq:P}
P_* Q_i = \frac{Q_{i+1} - Q_1}{\delta_i} .
\end{equation}
\end{lemma}
\begin{proof}[Proof of Lemma \ref{lem: pull back}]
	Set $\xi \coloneqq e^{\frac{2 \pi i }{d}}$. Observe that for a given $z \in \C$ and for a given $w \in \C$ such that $P(w) = z$, we have $\{ P(u)= z \} = \{ w,\xi w,\ldots,\xi^{d-1} w \}$ . Thus, for a given $i \in \{1,\ldots,k+m-1\}$,
\[\displaystyle \begin{array}{rcl}
	P_* Q_i (z) &=& \sum_{P(u) = z} \dfrac{1}{u -z_i}\dfrac{1}{\bigl(P'(u)\bigr)^2}\text{d} t^2\\
	 &=&\sum\limits_{j = 0}^{d-1} \dfrac{1}{\xi^j w - z_i} \dfrac{1}{\left(d \left(\xi^j w\right)^{d-1}\right)^2} \text{d}t^2\\
	 &=&\dfrac{1}{d^2 w^{2d-2}}\left(\sum\limits_{j=0}^{d-1}\dfrac{1}{\xi^{-j} w -  \xi^{-2j} z_i} \right) \text{d}t^2 .\\
\end{array}\]
Note that $\sum\limits_{j=0}^{d-1}\dfrac{1}{\xi^{-j} w -  \xi^{-2j} z_i} = \dfrac{ d z_i w^{d-2}}{w^d - z_i^d}\footnote{This equality is equivalent to the equality $\sum\limits_{j = 0}^{d-1} \frac{1}{\xi^{-j} \chi -\xi^{-2j}} = \frac{d \chi^{d-2}}{\chi^d - 1} $. The later follows from an elementary computation by comparing the partial fraction decomposition.}$. Therefore,
\[
\begin{array}{rcl}
P_* Q_i(z) & = & \dfrac{z_i}{dw^d(w^d - z_i^d)}\text{d}t^2\\
&=& \dfrac{1}{\delta_i} \dfrac{z_i^d}{w^d(w^d - z_i^d)}\text{d}t^2
\end{array}
\]
Since $w^d = z - z_1, z_{i}^d = z_{i+1} - z_1,$ we have
\[
P_* Q_i(z) = \dfrac{1}{\delta_i} \dfrac{z_{i+1} - z_1}{(z-z_1)(z- z_{i+1})} \text{d}t^2 = \frac{1}{\delta_i} \left(\frac{1}{z - z_{i+1}} -\frac{1}{z - z_1}\right) \text{d}t^2.
\]
Thus, $P_i Q_i = \dfrac{Q_{i+1} - Q_i}{\delta_{i}}$ \qedhere
	\end{proof}
The quadratic differentials $(Q_i)_{1 \le i \le k+m-1}$ span a vector space $\mathcal{Q}_P\subset \mathfrak{Q}(\C)$ of dimension $k+m-1$. According to Equation \eqref{eq:P}, this subspace is invariant by $P_*$. 
According to  Equations \eqref{eq:L} and \eqref{eq:P}, the linear map ${\iota :\mathcal{Q}_P\to {\mathcal{M}}_{k,m}}$ which sends $Q_i\in \mathfrak{Q}(\C)$ to $\omega_i\in \M_{k,m}$ is an isomorphism which conjugates ${P_*:\mathcal{Q}_P\to \mathcal{Q}_P}$ to $L_*:\M_{k,m}^*\to \M_{k,m}^*$

Since $\|P_* Q\|_U<\|Q\|_U$ for all $Q\in\mathcal{Q}_P$, the spectrum of $P_*:\mathcal{Q}_P\to \mathcal{Q}_P$ is contained in the unit disk. It follows that the spectrum of $L_*:\M_{k,m}^*\to \M_{k,m}^*$ is contained in the unit disk as required. 

\end{proof}
\subsubsection{Fixed points inside the post-critical set}\label{sec:fixed_point_inside}
We will now study the derivatives of moduli maps at fixed points which are inside the post-critical set. Let $\z \in PC(F_{k,m})$ be a fixed point of $F_{k,m}$ and let $(k',m')$ be the exact type of $\z$.

According to Lemma \ref{lem: position of fixed points}, $(k',m') \prec (k,m)$ and, by Proposition \ref{prop: partial order}, $\M_{k',m'} \subsetneq \M_{k,m}$ is invariant under $F_{k,m}$. Since $\M_{k',m'}$ is invariant under $D_{\z}F_{k,m}$, the vector space \[{\M_{k',m'}^0 = \{\omega \in \M_{k,m}^* \mid \omega|_{\M_{k',m'}} \equiv 0 \}},\]which is called \textit{the annihilator of }$ \M_{k',m'}$ in $\M_{k,m}$, is invariant under the transpose $L^*$ of $D_{z} F_{k,m}$ and we have the following decomposition
\begin{equation}\label{eq: decomposition of L}
\Spec L = \Spec(L|_{ \M_{k',m'}}) \cup \Spec\left( L^*|_{  \M_{k',m'}^0} \right).
\end{equation}
Moreover, according to Proposition \ref{prop: partial order}, we have
\[
L|_{\M_{k',m'}} = D_{\z} F_{k',m'}.
\]
Whence, by Proposition \ref{prop: fixed point outside}, $L|_{\M_{k',m'}}$ has only eigenvalues of modulus strictly greater than $1$. 
In order to describe $\Spec L$, we need to study  $\Spec\left( L^*|_{  \M_{k',m'}^0}\right)$. We will prove the following result.
\begin{proposition}\label{prop: eigenvalue inside}
	Let $(k,m) \in \N \times \N^*$ and $\z \in PC(F_{k,m})$ be a fixed point of $F_{k,m}$ of exact type $(k',m') \prec (k,m)$. Let $\lambda$ be the multiplier of the polynomial $P(t) = t^d + z_1 \in \C[t]$ along the cycle of $P^{\circ k'}(z_1)$. Then 
	\[
	\Spec \left( (D_{\z} F_{k,m})^*|_{  \M_{k',m'}^0} \right) =  
	\begin{cases} \{0\}&\text{if }k'=0\\
	\displaystyle \{ \mu \mid \mu^m = \lambda^{\frac{m}{m'}}, \mu^{m'} \neq \lambda \} &\text{if } k' \neq 0 .\end{cases}\]
\end{proposition}
The rest of this section is devoted to the proof of this proposition. To simplify the notation, we denote by $L^*$ the restriction of $(D_{\z} F_{k,m})^*$ to $\M_{k',m'}^0$. The study of the transpose $L^*\colon \M_{k',m'}^0 \to \M_{k',m'}^0$ is divided into two cases, $k' = 0 $ and $k' \neq 0$, and each case will be treated separately.
\begin{proof}[Proof of Proposition \ref{prop: eigenvalue inside} when $k'=0$] Since $(0,m') \preceq (k,m)$, $m'$ divides $k$ and $m$. It is enough to prove that $L^*: \M_{0,m'}^0 \to \M_{0,m'}^0$ is nilpotent. Recall that for $i \ge 1,$ the form $\omega_i\colon  \M_{k,m} \to \C$ is defined by $\omega_i(\boldsymbol{v}) =v_i$. For $i\ge 1$, set $\alpha_i\colon \M_{k,m} \to \C$ defined by 
\[
\alpha_i = \omega_i - \omega_{i+m'}.
\]
Recall that $\M_{0,m'}^0 = \{ \omega \colon \M_{k,m} \to \C \text{ such that } \omega|_{\M_{0,m'}} \equiv 0 \}$.
\begin{lemma}\label{lem: span of E}
We have $\M_{0,m'}^0 = \Span\left\{ \alpha_i, 1 \le i \le k+m  \right\}.$
\end{lemma}
\begin{proof}
By duality, it is equivalent to show that
\[
\M_{0,m'} = \bigcap\limits_{1 \le i \le k+m} \Ker \alpha_i.
\]

Assume $\boldsymbol{ v} \in \M_{0,m'}$. Then for all $j \ge 1,  v_{j} =  v_{j+m'}$. Given $i \in  \{1,\ldots,k+m\}$, we have
\[
\alpha_i (\boldsymbol{ v}) = \omega_i(\boldsymbol{ v}) - \omega_{i+m'}(\boldsymbol{ v}) =  v_{i} -  v_{i+m'} = 0.
\]
Hence $\M_{0,m'} \subseteq \bigcap\limits_{1 \le i \le k+m} \Ker \alpha_i$. 

Conversely, assume $\boldsymbol{ v} \in \bigcap\limits_{1 \le i \le k+m} \Ker \alpha_i$, i.e. for all $i \in  \{1,\ldots,k+m\}$, $v_i = v_{i+m'}$. In order to prove that $\boldsymbol{v} \in \M_{0,m'}$, we will prove that for all $j \ge k+m+1,  v_j =  v_{j+m'}$ and that $ v_{m'} = 0$. Given $j \geq k+m+1$, there exists an integer $j' \in  \{k+1,\ldots,k+m\}$ such that $j\equiv j' \mod{m}$. Since $ \boldsymbol{v} \in \M_{k,m}$, we have $v_{j} = v_{j'}$ and $v_{j+m'} = v_{j'+m'}$. Moreover, the fact that $\boldsymbol{ v} \in \Ker\alpha_{j'}$ implies that $v_{j'} = v_{j'+m'}$. Thus
\[
 v_{j} =  v_{j'} =  v_{j'+m'} =  v_{j+m'}.
\]
In order to conclude, we need to show that $ v_{m'} =0$. Note that the previous argument shows that $\boldsymbol{v}$ is a periodic sequence of period dividing $m'$. Since $m'$ divides $k$ and $m$,
\[
 v_{m'} =  v_{k} =  v_{k+m}.
\]
Since $\boldsymbol{ v} \in \M_{k,m},$ we have $ \b  v_{k+m} -v_k =0$ with $\b \neq 1$, whence \[v_{m'} =  v_{k} =  v_{k+m} =0.\]
\end{proof}
\begin{lemma}\label{lem: action of L periodic case}
We have $L^*\alpha_1= 0$ and for $i \ge 2$, $L^* \alpha_i = \delta_{i-1} \alpha_{i-1}$.
\end{lemma}
\begin{proof} According to Lemma \ref{lem: pull-back of L}, for all $i \ge 2$, 
	\[
	L^*\omega_i =  \delta_{i-1} \omega_{i-1} + L^*\omega_1.
	\]
Hence, if $i \ge 2$, 
\[
L^*\alpha_i = L^*( \omega_{i} - \omega_{i+m'})= \delta_{i-1}\omega_{i-1}  - \delta_{i+m'-1} \omega_{i+m'-1}.
\]
Since $\z \in \M_{0,m'}$, we have $\delta_{i-1} = d z_{i-1}^{d-1} = d z_{i+m'-1}^{d-1} = \delta_{i+m'-1}$. Hence 
\[
L^*\alpha_{i-1} = \delta_{i-1} (\omega_{i-1} - \omega_{i+m'-1}) = \delta_{i-1}  \alpha_{i-1}.
\]

If $i = 1$, since $m' \ge 1$, we have $1 + m' \ge 2$ so that
	\[
	L^*\omega_{1+m'} = \delta_{m'} \omega_{m'} + L^*\omega_1.
	\]
	Hence 
	\[
	L^*\alpha_{1} = L^*(\omega_1 - \omega_{1+m'}) = - \delta_{m'} \omega_{m'}.
	\]
Since $\z \in \M_{0,m'}$, we have $z_{m'} = 0$. Therefore, $\delta_{m'} = d z_{m'}^{d-1} = 0$ and $L^*\alpha_1 = 0$. \qedhere

\end{proof}
It follows from Lemma \ref{lem: span of E} and Lemma \ref{lem: action of L periodic case} that $L^*\colon\M_{0,m'} \to \M_{0,m'} $ is nilpotent.
\end{proof}
\begin{proof}[Proof of Proposition \ref{prop: eigenvalue inside} when $k' \neq 0$]
In this case, since $(k',m') \prec (k,m)$, we have \[k'=k \text{ and } m = pm'  \text{ with } p \ge 2.\] 
Let $\lambda$ be the multiplier of $P(t) = t^d + z_1$ at $P^{\circ k}(0)$. Note that, according to Proposition \ref{prop: critical orbit}, $\z$ is the critical orbit of $P$. Since $\z$ is preperiodic of preperiod $k >0$, the critical point $0$ of $P$ is preperiodic, i.e. $\lambda \neq 0$. We will show that 
\[
\Spec (L^*: \M_{k,m'}^0 \to \M_{k,m'}^0) = \{ \mu \mid \mu^m = \lambda^p, \mu^{m'} \neq \lambda \}.
\]
Given $j\in \Z/m\Z$, denote by ${\underline j}$ the representative of $j$ in $ \{k+1,\ldots,k+m\}$, define a linear form $\beta_j\colon\M_{k,m} \to \C$ by\[\beta_j \coloneqq \omega_{\underline j} - \omega_{\underline {j+m'}}.\]
Note that for all $j \in \Z/m\Z$, $\beta_{j}\colon\M_{k,m} \to \C$ is non trivial. Indeed, for a given $j \in \Z/m\Z$, define $\boldsymbol{u} \in\M_{k,m}$ by
\[
u_i = \begin{cases}
1 &\text{ if $ i \ge k+1$ and $i \equiv \ud{j} \mod{m}$}\\
\frac{1}{\b} & \text{ if $i =k$ and $\ud{j} =k+m$}\\
0&\text{otherwise}.
\end{cases}
\]
Since $m'< m$, $\beta_j(\boldsymbol{u}) = u_{\ud{j}} = 1 \neq 0 $.

We will show that these forms span $\M_{k,m'}^0 \subsetneq \M_{k,m}$ and use them to study the linear map $L^*: \M_{k,m'}^0 \to \M_{k,m'}^0$. The properties we need are provided by the following lemmas. 
\begin{lemma}\label{lem: structure of Omega}
We have $\M_{k,m'}^0 = \Span \{\beta_{ j},  j \in \Z/m\Z \}$.
\end{lemma}
\begin{proof}
By duality, it is equivalent to show that
\[
\M_{k,m'} =\bigcap\limits_{ j \in \Z/m\Z} \Ker \beta_{ j}.
\]

Assume $\boldsymbol{ v} \in \M_{k,m'}$. Then for all $i \ge k+1,  v_{i} =  v_{i+m'}$. Since $\underline{j} \in j $ and $m'$ divides $m$, we have $\underline{j}  \equiv \underline{j+m'} \mod{m'}$. Moreover, $\ud{j} \geq k+1$ and $\ud{j+m'}\geq k+1$, whence
\[
\beta_{ j}(\boldsymbol{v})  =  v_{\underline{j}} -  v_{\underline{j+m'}} =0.
\]
This shows that $\M_{k,m'} \subseteq \bigcap\limits_{ j \in \Z/m\Z} \Ker \beta_{ j}$.

Conversely, assume $\boldsymbol{ v} \in \bigcap\limits_{ j \in \Z/m\Z} \Ker \beta_{ j}$. We want to prove that for all integer $i\ge k+  1,  v_{i} =  v_{i+m'}$ and that $ \b  v_{k+m'} -v_k= 0$. First, assume $i \ge k+ 1$ and let $j \in \Z/m\Z$ be the congruence class of $i$. Since $\underline{j} \in j$, we have $i \equiv \ud{j} \mod{m}$. Moreover $\boldsymbol{ v} \in \Ker \beta_{j}$, and so
\begin{equation*}
 v_{\underline{j}} -  v_{\underline{j+m'}} = \beta_{j}(\boldsymbol{ v}) = 0.
\end{equation*}
From the fact that $\boldsymbol{ v} \in \M_{k,m}$, we therefore deduce that $v_i = v_{\underline{j}}$ and $v_{\underline{j+m'}} = v_{i+m'}.$ Thus,
\[ v_{i}  =   v_{i+m'}.\]
Second, let us show $\b v_{k+m'} -v_k = 0$. The previous argument shows that $\boldsymbol{v}$ is preperiodic of preperiod at most $k$ and of period dividing $m'$. Since $m'$ divides $ m$, we deduce that $ v_{k+m'} = v_{k+m}$. Since $\boldsymbol{ v} \in \M_{k,m}$, we have $ \b  v_{k+m} -v_k = 0$. Thus $\b  v_{k+m'}-v_k = 0$.
	\end{proof}

Thus, it is now important to understand the how $L^*$ acts on $\{\beta_j, j \in \Z/m\Z \}$. Given $j \in \Z/m\Z$, set 
\begin{equation*}
\sigma_j = \delta_{\ud{j}} =dz_{\ud{j}}^{d-1}.
\end{equation*}
\vspace{-0.7cm}
\begin{lemma}\label{lem: pull back preperiodic case}
For $j \in  \Z/m\Z $, we have $L^*\beta_{j} = \sigma_{{j-1}} \beta_{j-1}.$ 
\end{lemma}
\begin{proof}
For $j \in \Z/m\Z$, recall that $\underline{j}$ is the representative of $j$ in $\ob k+1, k+m \cb$. Let us first prove that for all $j \in \Z/m\Z$, 
\begin{equation}\label{eq: L^* case preper}
L^*\omega_{\underline{j}} = \sigma_{{j-1}}\omega_{\ud{j-1}} + L^*\omega_1.
\end{equation}
Indeed, if $\ud{j} = k+1$ then $\ud{j-1} = k+m$, whence $\sigma_{j-1} \om_{\ud{j-1}} = \delta_{k+m} \om_{k+m}$. According to Lemma \ref{lem: pull-back of L}, we have $L^* \omega_{k+1} = \delta_k \omega_{k} + L^* \omega_1$. Since $\z \in \M_{k,m}$, we have \[\delta_{k} = dz_{k}^{d-1} = d \b^{d-1}z_{k+m}^{d-1} = \b^{d-1}\delta_{k+m}.\]Moreover, $\om_{k} = \b\om_{k+m}$, whence $\delta_{k} \om_{k} =\b^d \delta_{k+m} \om_{k+m}$. Since $\b^d = 1,$ 
\[
L^*\om_{k+1} = \delta_k \om_k +L^* \om_1=  \delta_{k+m} \om_{k+m} + L^*\om_1.
\]
If $\ud{j} \neq k+1$ then $\ud{j-1} = \ud{j}-1$. According to Lemma \ref{lem: pull-back of L}, we have
\[
L^*\om_{\ud{j}} = \delta_{\ud{j}-1} \om_{\ud{j}-1} + L^*\om_{1} = \delta_{\ud{j-1}} \om_{\ud{j-1}} + L^* \om_1 =\sigma_{j-1} \om_{\ud{j-1}} + L^*\om_1.
\]
In any case, we have the equality \eqref{eq: L^* case preper}. Hence
\[
L^*\beta_j = L^*(\om_{\ud{j}} - \om_{\ud{j+m'}}) =\sigma_{j-1} \om_{\ud{j-1}} -\sigma_{j+m'-1} \om_{\ud{j+m'-1}}.
\]
Note that $\ud{j-1} $ and $\ud{j+m'-1}$, which are congruence modulo $m'$, are two integers at least $k+1$. Since $\z \in \M_{k,m'}$, we have $\sigma_{j-1} = \delta_{\ud{j-1}} = \delta_{\ud{j+m'-1}} = \sigma_{j+m'-1}$. Thus
\[
L^*\beta_j = \sigma_{j-1}(\om_{\ud{j-1}} - \om_{\ud{j+m'-1}}) = \sigma_{j-1} \beta_{j-1} .\qedhere
\]
\end{proof}
Recall that $\lambda$ is the multiplier of $P(t)= t^d +z_1$ at the periodic point $z_{k+1}$ of period $m' = \frac{m}{p}$.
\begin{lemma}\label{lem: sigma}
For all $j \in \Z/m\Z$, \[(L^*)^{  m'}( \beta_j) =\lambda \beta_{j-m'} \quad \text{ and } \quad (L^*)^{  m} (\beta_j) = \lambda^p \beta_j.\]
\end{lemma}
\begin{proof} The second equality is the straightforward consequence of the first one. Hence, it is enough to prove the first equality. According to Proposition \ref{prop: critical orbit}, $\z$ is the critical orbit of the polynomial $P(t) = t^d + z_1$, i.e. $z_i = P^{  i}(0)$. In particular, for any $j \in \Z/m\Z$, we have $P(z_{\ud{j-1}}) =z_{\ud{j}}$ and $P'(z_{\ud{j-1}})= dz_{\ud{j-1}}^{d-1} = \delta_{\ud{j-1}}$. Since $\z$ is of type $(k,m')$, the multiplier $\lambda$ of the cycle of $z_{\ud{j}}$ is 
	\[
	\lambda = \prod\limits_{i \in \ob 1,m' \cb} P'(z_{\ud{j-i}}) =\prod\limits_{i \in \ob 1,m' \cb} \delta_{\ud{j-i}} = \prod\limits_{i \in \ob 1,m' \cb} \sigma_{j-i}.
	\]
	Hence, by Lemma \ref{lem: pull back preperiodic case}, we have
	\[
	(L^*)^{ m'}(\beta_j) = \left(\prod\limits_{i \in \ob 1,m' \cb} \sigma_{j-i}\right) \beta_{j-m'} = \lambda \beta_{j-m'},
	\]
and 
	\[
	(L^*)^{  m}(\beta_j) = (L^*)^{m'p}(\beta_j)=\lambda^p \beta_j. \qedhere
	\]
\end{proof}
Let $\nu$ be a $m'$-th root of $\lambda$. Set
\[
T = \frac{L^*}{\nu}: \M_{k,m'}^0 \to \M_{k,m'}^0.
\]
We shall prove that $T$ is diagonalizable with simple eigenvalues and the eigenvalues of $T$ are $m$-th roots of unity except $1$. According to Lemma \ref{lem: sigma}, for all $j \in \Z/m\Z$, $T^{m'} (\beta_j) = \beta_{j-m'}$. In addition, 
\[\begin{array}{rcl}
\sum\limits_{n \in m'\Z/m\Z} \beta_{n} &=& \sum\limits_{n \in m'\Z/m\Z} (\om_{\ud{n}} - \om_{\ud{n+m'}}) =0
\\
&=& \sum\limits_{n \in m'\Z/m\Z} \omega_{\ud{n}} -\sum\limits_{n \in m'\Z/m\Z} \omega_{\ud{n+m'}} =0.
\end{array}\]
Recall that $p = \frac{m}{m'}$. Hence
\[
\beta_0 + T^{m'}(\beta_0) + \ldots + T^{m'(p-1)}(\beta_0) = 0
\]
Applying $m'-1$ times $T$ and adding the results, we deduce that
\[
\beta_0 + T(\beta_0) + T^2(\beta_0) + \ldots + T^{m'p-1}(\beta_0) = 0
\]
According to Lemma \ref{lem: structure of Omega} and Lemma \ref{lem: pull back preperiodic case}, the set $\{\beta_0,L^*(\beta_0),(L^*)^2(\beta_0),\ldots\}$ generates $\M_{k,m'}^0$. Hence $\{\beta_0,T(\beta_0),T^2(\beta_0),\ldots \}$ also generates $\M_{k,m'}^0$. Therefore,
\begin{equation}\label{eq: T}
\Id + T + T^2 + \ldots +T^{m-1} =0
\end{equation}
This means that the minimal polynomial of $T$ divides the polynomial $1+X+X^2+\ldots+X^{m-1}$. Consequently, $T$ is diagonalizable and the eigenvalues of $T$ are roots of unity which are not $1$. We now show that $T$ has only simple eigenvalues. Assume $\zeta \in \Spec T$. Let $v \in \M_{k,m'}$ be an eigenvector associated to $\zeta$. Set 
\[
H_\zeta = \frac{1}{m}\left(\Id+ \frac{T}{\zeta}+\ldots+\frac{T^{m-1}}{\zeta^{m-1}}\right).
\]
The equality \eqref{eq: T} implies that $T^m = \Id$. Additionally, $\zeta^m= 1$. Hence \[H_\zeta \circ \frac{T}{\zeta}= H_\zeta \quad \text{ so that } \quad H_\zeta \circ T^j = \zeta^j H_\zeta \quad \forall j \ge 1.
\] In addition, $\{ \beta_0,T(\beta_0),\ldots \}$ generates $\M_{k,m'}^0$, hence
\[
\Imm H_\zeta \subseteq \Span H_\zeta(\beta_0).
\]
Note that $H_\zeta(v) = v$. Hence
\[
v \subseteq \Imm H_\zeta \subseteq \Span H_\zeta(\beta_0).
\]
Thus the eigenspace associated to the eigenvalue $\zeta$ of $T$ has dimension $1$, i.e. $T$ has only simple eigenvalues.  

Since $T = \frac{L^*}{\nu}$, $L^*$ is diagonalizable with simple eigenvalues which are $m$-th roots of $\nu^m = \lambda^p$. In addition, $1$ is not an eigenvalue of $T$ hence $\nu$ is not an eigenvalue of $L^*$. Since $\nu$ is an arbitrary $m'$-th root of $\lambda$, we deduce that 
\[
\Spec L^* \subseteq \{ \mu^m = \lambda^p, \mu^{m'} \neq \lambda \}.
\]
Since $L^*$ has only simple eigenvalues, $\# \Spec L^* = \dim \M_{k,m'}^0 = m-m'$. Hence
\[
\Spec (L^*: \M_{k,m'}^0 \to \M_{k,m'}^0) = \{ \mu^m = \lambda^p, \mu^{m'} \neq \lambda \} \qedhere
\]
\end{proof}

\begin{proof}[Proof of Theorem \ref{thm:theoremA'}]
	According to Lemma \ref{prop: FKm and GKm}, $F_{k,m}$ and $G_{k,m}$ are conjugate. Hence it is enough to consider an eigenvalue $\mu$ of $F_{k,m}$ at a fixed point $\z \in \M_{k,m}$. Denote by $(k',m')$ the exact type of $\z$. Then by Proposition \ref{prop: critical orbit}, we have $(k',m') \preceq (k,m)$, or by Definition \ref{def: partial order}, we have
	\[
	m' \mid m, (\text{either }  k' =0 \text{ or } k' = k,m' \mid k).
	\]
	
Regarding $\mu$ as an eigenvalue of $\Fkm$ at $\z$. We recall argument at the beginning of \ref{sec:fixed_point_inside}, according to Lemma \ref{lem: position of fixed points}, $(k',m') \preceq (k,m)$ and, by Proposition \ref{prop: partial order}, $\M_{k',m'} \subseteq \M_{k,m}$ is invariant under $F_{k,m}$. Since $\M_{k',m'}$ is invariant under $D_{\z}F_{k,m}$, the annihilator $\M_{k',m'}^0$ of $ \M_{k',m'}$ in $\M_{k,m}$ is invariant under the transpose $L^*$ of $D_{z} F_{k,m}$ and we have the following decomposition
\begin{equation}
\Spec L = \Spec(L|_{ \M_{k',m'}}) \cup \Spec\left( L^*|_{  \M_{k',m'}^0} \right).
\end{equation}
Moreover, according to Proposition \ref{prop: partial order}, we have
\[
L|_{\M_{k',m'}} = D_{\z} F_{k',m'}.
\]

If $\mu \in  \Spec(L|_{ \M_{k',m'}})$, note that $L|_{\M_{k',m'}} = D_{\z} F_{k',m'}$ and $\z \notin PC_{k',m'}$, thus $\mu$ is an eigenvalue of $G_{k',m'}$ at a fixed point outside its post-critical set.

If $\mu \in \Spec\left( L^*|_{  \M_{k',m'}^0} \right)$, then we are done by Proposition \ref{prop: eigenvalue inside}.
\end{proof}
\bibliography{ref}
\bibliographystyle{alpha}
\end{document}